\documentclass[11pt]{amsart}
\usepackage{amssymb}
\usepackage{amsmath}
\usepackage{latexsym}
\usepackage{microtype}
\usepackage{tikz}
\usepackage{hyperref}
\hypersetup{colorlinks=true,citecolor=blue,linkcolor=blue}

\usepackage[margin=1in,marginparwidth=0.8in, marginparsep=0.1in]{geometry}

\newtheorem{theorem}{Theorem}[section]
\newtheorem{claim*}[theorem]{Claim}

\newtheorem{definition}[theorem]{Definition}
\newtheorem{lemma}[theorem]{Lemma}
\newtheorem{proposition}[theorem]{Proposition}
\newtheorem{scholium}[theorem]{Scholium}
\theoremstyle{remark}
\newtheorem{example}[theorem]{Example}
\newtheorem{remark}[theorem]{Remark}

\numberwithin{equation}{section}

\usepackage{cleveref}
\crefname{theorem}{Theorem}{Theorems}
\crefname{proposition}{Proposition}{Propositions}
\crefname{section}{Section}{Sections}
\crefname{subsection}{Section}{Sections}
\crefname{definition}{Definition}{Definitions}
\crefname{remark}{Remark}{Remarks}
\crefname{lemma}{Lemma}{Lemmas}
\crefname{corollary}{Corollary}{Corollaries}
\crefname{example}{Example}{Examples}
\crefname{figure}{Figure}{Figures}

\newcommand{\QQ}{\mathbb{Q}}
\newcommand{\RR}{\mathbb{R}}
\newcommand{\ZZ}{\mathbb{Z}}

\newcommand{\cA}{\mathcal{A}}

\newcommand{\fd}{\mathfrak{d}}
\newcommand{\fD}{\mathfrak{D}}
\newcommand{\fm}{\mathfrak{m}}
\newcommand{\fp}{\mathfrak{p}}

\newcommand{\bfd}{{\bf d}}
\newcommand{\bfg}{{\bf g}}

\newcommand{\Hom}{\operatorname{Hom}}
\newcommand{\Mono}{\operatorname{Mono}}
\newcommand{\Spec}{\operatorname{Spec}}

\title[The Greedy Basis equals the Theta Basis]{The Greedy Basis \\ Equals The Theta Basis: \\ A Rank Two Haiku}
\author[Cheung]{Man Wai Cheung}
\address[Man Wai Cheung]{University of California, San Diego}
\email{mwc31@cam.ac.uk}

\author[Gross]{Mark Gross}
\address[Mark Gross]{University of Cambridge}
\email{mg475@dpmms.cam.ac.uk}

\author[Muller]{Greg Muller}
\address[Greg Muller]{University of Michigan}
\email{morilac@umich.edu}

\author[Musiker]{Gregg Musiker}
\address[Gregg Musiker]{University of Minnesota}
\email{musiker@math.umn.edu}

\author[Rupel]{Dylan Rupel}
\address[Dylan Rupel]{University of Notre Dame}
\email{drupel@nd.edu}

\author[Stella]{Salvatore Stella}
\address[Salvatore Stella]{Universit\`a degli studi di Roma ``La Sapienza''}
\email{stella@mat.uniroma1.it}

\author[Williams]{Harold Williams}
\address[Harold Williams]{The University of Texas at Austin}
\email{hwilliams@math.utexas.edu}

\begin{document}

\begin{abstract}
We prove the equality of two canonical bases of a rank 2 cluster algebra, the greedy basis of Lee-Li-Zelevinsky and the theta basis of Gross-Hacking-Keel-Kontsevich.

\bigskip

\begin{center}
{Dedicated to the memory of Andrei Zelevinsky.}
\end{center}

\end{abstract}

\maketitle

\section{Introduction}

Cluster algebras are commutative rings with partial bases of a special form, originally discovered in the context of dual canonical bases in Lie theory \cite{FZ}.  Their axiomatics encapsulates the fact that many kinds of canonical bases in nature have large subsets which are governed by a uniform combinatorics.  Elements of these subsets are monomials in distinguished elements called cluster variables, which are grouped into overlapping collections called clusters.  Each cluster has an associated skew-symmetrizable matrix and the entire cluster algebra can be reconstructed recursively from any particular cluster along with this matrix.

A fundamental issue in the theory is understanding natural completions of the partial basis of cluster monomials to a full basis of the cluster algebra.  Depending on the context, this question can be analyzed from a wide range of perspectives drawn from representation theory, geometry, combinatorics, and mathematical physics \cite{Dup,KQ,FG,MSW,BZ,Rup2,GMN}.  In general, one expects any cluster algebra to admit several natural bases related in potentially subtle ways.  A basic example of this is the relationship between the dual canonical and dual semicanonical bases of the coordinate ring of the positive unipotent subgroup of a simple algebraic group \cite{GLS}.  This example also illustrates that, in general, even determining whether or not two constructions of canonical bases in a cluster algebra lead to the same result is nontrivial.  The purpose of the present paper is to compare two such constructions for cluster algebras associated to $2\times2$ skew-symmetrizable matrices. 

The first basis we consider is the greedy basis of \cite{LLZ}.  Every cluster algebra is contained in the ring of Laurent polynomials in the cluster variables of any of its clusters.  The recently-confirmed positivity conjecture, proved in the rank 2 case in \cite{LS,Rup} and in the general case in \cite{LS2,GHKK}, asserts that the coefficients of the Laurent expansion of any cluster variable are positive integers.  The greedy basis is defined so that all of its elements, not just cluster variables, have positive Laurent expansions in any cluster and that the coefficients of any such Laurent expansion are as small as possible.  The resulting coefficients turn out to enumerate combinatorial objects called compatible pairs related to maximal Dyck paths.

The second basis we consider is the theta basis of \cite{GHKK}.  Unlike the
greedy basis it is defined for cluster algebras of arbitrary rank.  In fact,
this basis is a special case of a much more general construction based on two
concepts. The first is that of scattering diagram introduced in \cite{KS} in two
dimensions and in \cite{GS} in all dimensions.   This diagram
encodes the relations among cluster transformations and also among elements of the
tropical vertex group.  The second is a combinatorial
notion of broken line, introduced in \cite{G10} with their theory further
developed in \cite{CPS} and then \cite{GHK11}.  The coefficients of Laurent
expansions of theta basis elements enumerate broken lines.  These are piecewise-linear paths in a tropicalization of the cluster variety
whose points of non-linearity lie along the scattering diagram.   Morally
broken lines capture the geometry of holomorphic disks in the mirror cluster variety.

Our main result is the following.
\begin{theorem}
Let $\mathcal{A}$ be a rank 2 cluster algebra.  The greedy and theta bases of $\mathcal{A}$ coincide.
\end{theorem}

The proof is based on an analysis of exactly which monomials may appear in elements of the theta basis.  It can be shown that elements of the greedy basis are essentially determined by which coefficients of their Laurent expansion are nonzero.  That is, if an element of $\mathcal{A}$ has the same support as a greedy basis element in any particular Laurent expansion, it must in fact coincide with that element up to a scalar.  Thus to show that elements of the theta basis are elements of the greedy basis, it suffices to establish certain bounds on the behavior of broken lines rather than explicitly enumerating them.

The organization of the paper is as follows.  In sections 2 and 3, we review the basic definitions and properties of the greedy and theta bases, respectively.  The natural parametrizing sets of the two bases, the $\bfd$-vectors and $\bfg$-vectors, are distinct and we explain in section 4 how to relate them.  This determines a bijection between the two bases and we show in section 5 that the basis elements mapped to each other by this bijection actually coincide, proving the main theorem.

\textsc{Acknowledgements} This paper is the result of a working group on scattering diagrams at the 2014 AMS Mathematics Research Community on Cluster Algebras in Snowbird, UT, which accounts for the large number of authors.  We thank the AMS and in particular Ellen Maycock and Donna Salter for helping facilitate such an enjoyable and productive environment.  We thank Maria Angelica Cueto, who was also member of our working group but declined to be named as a coauthor.  We also thank the other organizers Gordana Todorov, Michael Gekhtman, and David Speyer for making the workshop possible. Man Wai Cheung would like to thank University of Cambridge for hosting her during the completion of this paper.  A portion of this work was completed while Dylan Rupel was a research instructor at Northeastern University and Salvatore Stella was a post doctoral research scholar at North Carolina State University.

\section{Rank 2 cluster algebras and their greedy bases}
Fix positive integers $b$ and $c$.  Consider rational functions $x_k\in\QQ(x_1,x_2)$ indexed by $k\in\ZZ$ and defined recursively by 
\begin{equation}\label{eq:exchangerelations}
  x_{k-1}x_{k+1}=
  \begin{cases}
    x_k^b+1 & \text{if $k$ is odd;}\\
    x_k^c+1 & \text{if $k$ is even.}
  \end{cases}
\end{equation}
These functions are called \emph{cluster variables} and the \emph{cluster
algebra} $\cA(b,c)$ is the $\ZZ$-subalgebra of $\QQ(x_1,x_2)$ which they
generate.  Each pair $\{x_k, x_{k+1}\}$ is called a \emph{cluster} and a monomial in the variables of a cluster is called a \emph{cluster monomial}.  Later, we will fix a rank 2 lattice $M$ together with an algebra isomorphism $\ZZ[M] \cong \ZZ[x_1^{\pm1},x_2^{\pm1}]$, $x^m\mapsto x_1^{m_1}x_2^{m_2}$ for $m \in M$.  This induces a lattice isomorphism $M \cong \ZZ^2$, $m \mapsto (m_1,m_2)$.

An essential feature of the relations \eqref{eq:exchangerelations} is that they imply $\cA(b,c)$ is actually a subalgebra of $\ZZ[x_1^{\pm1},x_2^{\pm1}]$, rather than merely a subalgebra of $\QQ(x_1,x_2)$.
\begin{theorem}\label{th:Laurent}
\cite[Theorem~3.1]{FZ} Given any cluster variable $x_j$, we have $x_j\in\ZZ[x_k^{\pm1},x_{k+1}^{\pm1}]$ for every $k \in \ZZ$.
\end{theorem}

We will denote by $\ZZ_{\ge0}[x_k^{\pm1},x_{k+1}^{\pm1}]$ the subspace of Laurent polynomials with positive coefficients.  An element of $\QQ(x_1,x_2)$ is a \emph{universal Laurent polynomial} (resp.  \emph{positive universal Laurent polynomial}) if it is contained in 
 $\ZZ[x_k^{\pm1},x_{k+1}^{\pm1}]$ (resp. $\ZZ_{\ge0}[x_k^{\pm1},x_{k+1}^{\pm1}]$) for every $k \in \ZZ$.  A primary result of \cite{BFZ}, specialized to the rank 2 setting, states that $\cA(b,c)$ is precisely the set of univeral Laurent polynomials in $\QQ(x_1,x_2)$.

\begin{theorem}
\cite{LS,Rup} Each cluster variable of $\cA(b,c)$ is positive.
\end{theorem}

An element of $\ZZ[x_1^{\pm1},x_2^{\pm1}]$ is called \emph{pointed at
$(a_1,a_2)\in\ZZ^2$} if it can be written in the form
\[
  x_1^{-a_1}x_2^{-a_2}\sum\limits_{p_1,p_2\ge0}c(p_1,p_2)x_1^{bp_1}x_2^{cp_2},
\]
where $c(p_1,p_2)\in\ZZ$ with $c(0,0)=1$.   
\begin{proposition}\cite[Proposition~1.5]{LLZ}
  Let $z$ be pointed at $(a_1,a_2)\in\ZZ^2$ and suppose
  $z\in\ZZ_{\ge0}[x_0^{\pm1},x_1^{\pm1}]
  \cap\ZZ_{\ge0}[x_1^{\pm1},x_2^{\pm1}]
  \cap\ZZ_{\ge0}[x_2^{\pm1},x_3^{\pm1}]$.
  Then the pointed coefficients $c(p_1,p_2)$ satisfy the following recursive
  inequality:
  \begin{align}
    \label{eq:consecutive positivity inequality}
    c(p_1,p_2)\ge\max\bigg(
    &\sum\limits_{k=1}^{p_1} (-1)^{k-1}c(p_1-k,p_2){a_2-cp_2+k-1\choose k},\\
    \nonumber&\sum\limits_{ j =1}^{p_2} (-1)^{ j -1}c(p_1,p_2- j ){a_1-bp_1+ j -1\choose  j }\bigg).
  \end{align}
\end{proposition}
A positive element of $\cA(b,c)$ is called \emph{indecomposable} if it cannot
be written as a sum of two positive elements.  In the search for positive bases
of $\cA(b,c)$ one is naturally led to investigate the indecomposable positive
elements.  A sufficient condition for a positive pointed element to be
indecomposable is the inequality \eqref{eq:consecutive positivity inequality}
being an equality.  It turns out that this requirement alone uniquely determines
a collection of elements of $\cA(b,c)$ with nice properties.

\begin{theorem}\label{th:greedy}\cite[Theorem~1.7]{LLZ}
  For any $(a_1,a_2)\in\ZZ^2$ there exists a unique indecomposable positive
  element $x[a_1,a_2]\in\cA(b,c)$ which is pointed at $(a_1,a_2)$ and whose
  pointed coefficients satisfy the recursion
  \begin{align}
    \label{eq:greedy recursion}
    c(p_1,p_2)=\max\bigg(
    &\sum\limits_{k=1}^{p_1} (-1)^{k-1}c(p_1-k,p_2){a_2-cp_2+k-1\choose k},\\
    \nonumber&\sum\limits_{ j =1}^{p_2} (-1)^{ j -1}c(p_1,p_2- j ){a_1-bp_1+ j -1\choose  j }\bigg).
  \end{align}
  Moreover, the collection $\{x[a_1,a_2]:(a_1,a_2)\in\ZZ^2\}$ is a basis of
  $\cA(b,c)$ which contains the cluster monomials and is independent of the
  choice of an initial cluster.
\end{theorem}
We will call $x[a_1,a_2]$ the \emph{greedy element} pointed at $(a_1,a_2)$ and
call $\{x[a_1,a_2]:(a_1,a_2)\in\ZZ^2\}$ the \emph{greedy basis} of $\cA(b,c)$. 
In view of the definition of pointed elements, $(a_1,a_2)$ is the
\emph{$\bfd$-vector} of $x[a_1,a_2]$; we refer to \cite{FZ4} for the
 definitions and basic properties of $\bfd$-vectors.
In order to better connect with the scattering diagram approach from Section 3, we now switch our point of view and consider ordinary support rather than pointed support.  Given a Laurent polynomial $f=\sum_{m\in M}c_mx^m$ in $\ZZ[x_1^{\pm1},x_2^{\pm1}]$, the \emph{support} of $f$ is the set
\[ \{m\in M \mid c_m\neq 0\}. \]

\begin{theorem}\label{th:Newton polygons}\cite[Proposition~4.1]{LLZ}, \cite[Corollary~3.5]{LLZ2}
  For $(a_1,a_2)\in\ZZ^2$, the smallest (possibly degenerate) lattice quadrilateral $R_{a_1,a_2}$ containing the support of $x[a_1,a_2]$ is determined as follows.
  \begin{enumerate}
    \item If $a_1 \leq 0$ and $a_2 \leq 0$, then $R_{a_1,a_2} = \{(-a_1,-a_2)\}$.
    \item If $a_1 \leq 0 < a_2$, then $R_{a_1,a_2} = \{(p_1,-a_2): -a_1\leq p_1\leq -a_1+ba_2\}$.
    \item If $a_2 \leq 0 < a_1$, then $R_{a_1,a_2} = \{(-a_1,p_2): -a_2 \leq p_2 \leq -a_2+ca_1\}$.
    \item If $0<ba_2\leq a_1$, then $R_{a_1,a_2}=\left\{(p_1,p_2): -a_1\le p_1\le -a_1+ba_2,\; -a_2\le p_2\le -a_2-cp_1\right\}$.
    \item If $0<ca_1\leq a_2$, then $R_{a_1,a_2}=\left\{(p_1,p_2): -a_1\le p_1\le -a_1-bp_2,\; -a_2\le p_2\le -a_2+ca_1\right\}$.
    \item If $0 < a_1 < ba_2$ and $0 < a_2 < ca_1$, then
    $$\aligned
    R_{a_1,a_2}=&\bigg\{(p_1,p_2)\bigg|\; -a_1\le p_1<0,\; -a_2\le p_2<\Big(\frac{a_2}{a_1}-c\Big)p_1\bigg\}\\
    &\bigcup\bigg\{(p_1,p_2): -a_1\le p_1<\Big(\frac{a_1}{a_2}-b\Big)p_2,\; -a_2\le p_2<0\bigg\}\\
    &\hspace{1.35em}\bigcup\Big\{(-a_1+ba_2,-a_2),(-a_1,-a_2+ca_1)\Big\}.
    \endaligned
    $$
  \end{enumerate}
  Moreover, if $z\in\cA(b,c)$ is pointed at $(a_1,a_2)$ with support contained in $R_{a_1,a_2}$, then $z=x[a_1,a_2]$. 
\end{theorem}
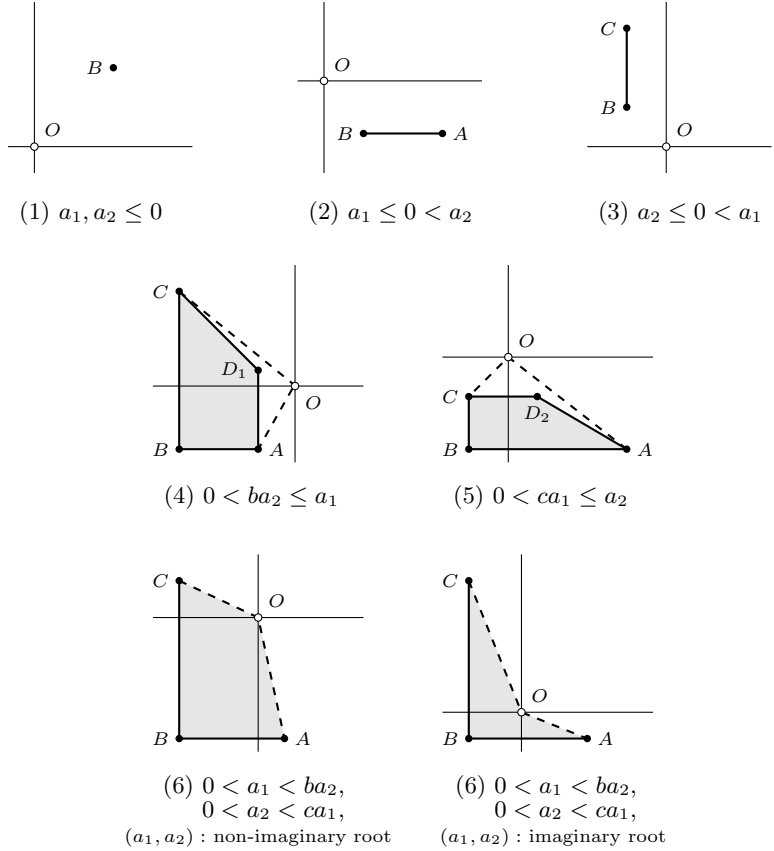
\begin{figure}[h!]
  \centering
  \begin{tikzpicture}[scale=.7]
    \begin{scope}[shift={(-2.75, 5.5)}]
      \usetikzlibrary{patterns}
      \draw[-] (-0.5,0.25) -- (3,0.25);
      \draw[-] (0,-0.25) -- (0,3);
      \fill[white] (0,0.25) circle (2pt);
      \draw (0,0.25) circle (2pt);
      \draw (0,0.25) node[anchor=south west] {\tiny $O$};
      \fill (1.5,1.75) circle (2pt);
      \draw (1.5,1.75) node[anchor=east] {\tiny $B$};
      \draw (1.07,-.6) node[anchor=north] {\footnotesize (1) $a_1,a_2\le0$};
    \end{scope}
    \begin{scope}[shift={(2.75, 5.5)}]
      \usetikzlibrary{patterns}
      \draw[-] (-0.5,1.5) -- (3,1.5);
      \draw[-] (0,-0.25) -- (0,3);
      \fill[white] (0,1.5) circle (2pt);
      \draw (0,1.5) circle (2pt);
      \draw (0,1.5) node[anchor=south west] {\tiny $O$};
      \fill (2.25,0.5) circle (2pt);
      \draw (2.25,0.5) node[anchor=west] {\tiny$A$};
      \fill (0.75,0.5) circle (2pt);
      \draw (0.75,0.5) node[anchor=east] {\tiny$B$};
      \draw [thick] (0.75,0.5) -- (2.25,0.5);
      \draw (1.27,-.6) node[anchor=north] {\footnotesize (2) $a_1\le 0<a_2$};
    \end{scope}
    \begin{scope}[shift={(8.25,5.5)}]
      \usetikzlibrary{patterns}
      \draw [thick] (0.25,1) -- (0.25,2.5);
      \draw[-] (-0.5,0.25) -- (3,0.25);
      \draw[-] (1,-0.25) -- (1,3);
      \fill[white] (1,0.25) circle (2pt);
      \draw (1,0.25) circle (2pt);
      \draw (1,0.25) node[anchor=south west] {\tiny $O$};
      \fill (0.25,1) circle (2pt);
      \draw (0.25,1) node[anchor=east] {\tiny$B$};
      \fill (0.25,2.5) circle (2pt);
      \draw (0.25,2.5) node[anchor=east] {\tiny$C$};
      \draw (1.27,-.6) node[anchor=north] {\footnotesize (3) $a_2\le0<a_1$};
    \end{scope}
    \begin{scope}[shift={(0,0)}]
      \usetikzlibrary{patterns}
      \fill[black!10] (0,3)--(1.5,1.5)--(1.5,0)--(0,0)--(0,3);
      \draw[thick] (0,3)--(1.5,1.5)--(1.5,0)--(0,0)--(0,3);
      \draw[dashed,thick](0,3)--(2.15,1.25) (2.15,1.15)--(1.5,0);
      \draw[-] (-0.5,1.2) -- (3.5,1.2);
      \draw[-] (2.2,-0.25) -- (2.2,3.5);
      \fill[white] (2.2,1.2) circle (2pt);
      \draw (2.2,1.2) circle (2pt);
      \draw (2.2,1.2) node[anchor=north west] {\tiny$O$};
      \fill (1.5,0) circle (2pt);
      \draw (1.5,0) node[anchor=west] {\tiny$A$};
      \fill (0,0) circle (2pt);
      \draw (0,0) node[anchor=east] {\tiny$B$};
      \fill (0,3) circle (2pt);
      \draw (0,3) node[anchor=east] {\tiny$C$};
      \fill (1.5,1.5) circle (2pt);
      \draw (1.5,1.5) node[anchor=east] {\tiny$D_1$};
      \draw (1.37,-0.5) node[anchor=north] {\footnotesize (4) $0<ba_2\leq a_1$};
    \end{scope}
    \begin{scope}[shift={(5.5,0)}]
      \usetikzlibrary{patterns}
      \fill[black!10] (3,0)--(1.3,1)--(0,1)--(0,0)--(3,0);
      \draw[thick] (3,0)--(1.3,1)--(0,1)--(0,0)--(3,0);
      \draw[dashed,thick](3,0)--(0.75,1.75)--(0,1);
      \draw[-] (-0.5,1.75) -- (3.5,1.75);
      \draw[-] (0.75,-0.25) -- (0.75,3.5);
      \fill[white] (0.75,1.75) circle (2pt);
      \draw (0.75,1.75) circle (2pt);
      \draw (0.75,1.75) node[anchor=south west] {\tiny$O$};
      \fill (3,0) circle (2pt);
      \draw (3,0) node[anchor=west] {\tiny$A$};
      \fill (0,0) circle (2pt);
      \draw (0,0) node[anchor=east] {\tiny$B$};
      \fill (0,1) circle (2pt);
      \draw (0,1) node[anchor=east] {\tiny$C$};
      \fill (1.3,1) circle (2pt);
      \draw (1.3,1) node[anchor=north] {\tiny$D_2$};
      \draw (1.37,-.5) node[anchor=north] {\footnotesize (5) $0<ca_1\leq a_2$};
    \end{scope}
    \begin{scope}[shift={(0,-5.5)}]
      \usetikzlibrary{patterns}
      \fill[black!10]  (0,3)--(1.5,2.3)--(2,0)--(0,0)--(0,3);
      \draw[thick] (0,3)--(0,0)--(2,0); 
      \draw[dashed,thick] (0,3)--(1.5,2.3)--(2,0);
      \draw[-] (-0.5,2.3) -- (3.5,2.3);
      \draw[-] (1.5,-0.25) -- (1.5,3.5);
      \fill[white] (1.5,2.3) circle (2pt);
      \draw (1.5,2.3) circle (2pt);
      \draw (1.5,2.3) node[anchor=south west] {\tiny$O$};
      \fill (2,0) circle (2pt);
      \draw (2,0) node[anchor=west] {\tiny$A$};
      \fill (0,0) circle (2pt);
      \draw (0,0) node[anchor=east] {\tiny$B$};
      \fill (0,3) circle (2pt);
      \draw (0,3) node[anchor=east] {\tiny$C$};
      \draw (1.42,-.5) node[anchor=north] {\footnotesize (6) $0<a_1<ba_2$,};
      \draw (1.46,-1) node[anchor=north] {\footnotesize \hspace{10pt} $0<a_2<ca_1$,};
      \draw (1.17,-1.5) node[anchor=north] {\footnotesize \hspace{10pt} \tiny{$(a_1,a_2):$ non-imaginary root}};
    \end{scope}
    \begin{scope}[shift={(5.5,-5.5)}]
      \usetikzlibrary{patterns}
      \fill[black!10]  (0,3)--(1,0.5)--(2.25,0)--(0,0)--(0,3);
      \draw[thick] (0,3)--(0,0)--(2.25,0);
      \draw[dashed,thick] (0,3)--(1,0.5)--(2.25,0);
      \draw[-] (-0.5,0.5) -- (3.5,0.5);
      \draw[-] (1,-0.25) -- (1,3.5);
      \fill[white] (1,0.5) circle (2pt);
      \draw (1,0.5) circle (2pt);
      \draw (1,0.5) node[anchor=south west] {\tiny$O$};
      \fill (2.25,0) circle (2pt);
      \draw (2.25,0) node[anchor=west] {\tiny$A$};
      \fill (0,0) circle (2pt);
      \draw (0,0) node[anchor=east] {\tiny$B$};
      \fill (0,3) circle (2pt);
      \draw (0,3) node[anchor=east] {\tiny$C$};
      \draw (1.47,-.5) node[anchor=north] {\footnotesize (6) $0<a_1<ba_2$,};
      \draw (1.51,-1) node[anchor=north] {\footnotesize \hspace{10pt} $0<a_2<ca_1$,};
      \draw (1.27,-1.5) node[anchor=north] {\footnotesize \hspace{10pt} \tiny{$(a_1,a_2):$ imaginary root}};
    \end{scope}
  \end{tikzpicture}
  \caption{Consider the points $O=(0,0)$, $A=(-a_1+ba_2,-a_2)$, $B=(-a_1,-a_2)$, $C=(-a_1,-a_2+ca_1)$, $D_1=(-a_1+ba_{2},ca_{1}-(bc+1)a_{2})$, and $D_2=(ba_{2}-(bc+1)a_{1},-a_2+ca_{1})$.  Then the support region of $x[a_1,a_2]$ can be visualized as above.} 
  \label{fig:greedy polytopes}
\end{figure}
\begin{proof}
  The first claim is the content of \cite[Proposition~4.1]{LLZ} and \cite[Corollary~3.5]{LLZ2}.

  Suppose $z\in\cA(b,c)$ is pointed at $(a_1,a_2)$ and that the support of $z$ is contained in $R_{a_1,a_2}$.  Suppose $z\ne x[a_1,a_2]$.  Then there exists a monomial $x_1^{-a'_1}x_2^{-a'_2}$ appearing in $z$ with a different coefficient than in $x[a_1,a_2]$.  Our assumptions on $z$ imply for any such monomial that we have $a'_1<a_1$ or $a'_2<a_2$.  Choose a monomial with $(a'_1,a'_2)$ minimal in lexicographic order.  Then in the greedy basis expansion of $z$ the element $x[a'_1,a'_2]$ must appear with nonzero coefficient. 

  Below we refer to the points $O,A,B,C$ from Figure~\ref{fig:greedy polytopes}.  To reach a contradiction, there are two cases to consider.
  \begin{itemize}
    \item If $(-a'_1,-a'_2)$ lies on or North of the line segment $OB$, i.e. $a_1a'_2\le a'_1a_2$, then we consider the point $C'=(-a'_1,-a'_2+ca'_2)$ at the Northern boundary of the support region $R_{a'_1,a'_2}$ of $x[a'_1,a'_2]$ and compare with the line segment $OC$.  In this case, we have
    \[-a_1(-a'_2+ca'_1)=a_1a'_2-ca_1a'_1\le a'_1a_2-ca_1a'_1=-a'_1(-a_2+ca_1)\]
    and thus $C'$ lies on or North of $OC$.  If $C'$ is North of $OC$ or $C'\ne C$ is on $OC$, then it lies outside $R_{a_1,a_2}$ which is impossible.  Thus we must have $C'=C$, but this implies $(a'_1,a'_2)=(a_1,a_2)$ which clearly must be false.
    \item If $(-a'_1,-a'_2)$ lies on or East of the line segment $OB$, i.e. $a'_1a_2\le a_1a'_2$, then we consider the point $A'=(-a'_1+ba'_2,-a'_2)$ at the Eastern boundary of the support region $R_{a'_1,a'_2}$ of $x[a'_1,a'_2]$ and compare with the line segment $OA$.  In this case, we have
    \[(-a'_1+ba'_2)(-a_2)=a'_1a_2-ba_2a'_2\le a_1a'_2-ba_2a'_2=(-a_1+ba_2)(-a'_2)\]
    and thus $A'$ lies on or East of $OA$.  If $A'$ is East of $OA$ or $A'\ne A$ is on $OA$, then it lies outside $R_{a_1,a_2}$ which is impossible.  Thus we must have $A'=A$, but this implies $(a'_1,a'_2)=(a_1,a_2)$ which is clearly false.
  \end{itemize}
  It follows that $z=x[a_1,a_2]$.
\end{proof}
The proof of Theorem~\ref{th:Newton polygons} actually establishes the following stronger result, which never uses the special `pointed' form, thus allowing for support anywhere in the region $R_{a_1,a_2}$.
\begin{scholium}\label{sch}
  If $z\in\cA(b,c)$ is any element containing the monomial $x_1^{-a_1}x_2^{-a_2}$ with coefficient 1 and whose support is contained in the half-open quadrilateral $OABC$ from Figure~\ref{fig:greedy polytopes} associated to $(a_1,a_2)$, then $z=x[a_1,a_2]$. 
\end{scholium}

\begin{remark}\label{remark:Dyck}
The existence of integers $c(p_1,p_2)$ satisfying the recursive equations \eqref{eq:greedy recursion}, and thus the existence of the greedy basis itself, is quite non-trivial.  The authors of \cite{LLZ} characterize each $c(p_1,p_2)$ as the solution to an enumerative problem; specifically, the number of certain `compatible pairs of edges' inside a type of lattice path called a `maximal Dyck path'.

This enumerative description not only establishes the existence of the greedy basis, but shows that the coefficients $c(p_1,p_2)$ are manifestly non-negative.  Finding naturally-defined bases for cluster algebras whose elements have positive coefficients has been one of the core goals of the theory since its inception.  \end{remark}

\section{Scattering diagrams and broken lines}

In this section we describe the theta basis of a rank 2 cluster algebra.  
We use \cite{GHKK} as a reference, adapting 
the notation to the rank 2 situation.

Recall from the previous section the lattice $M \cong \ZZ^2$ such that $\ZZ[M] \cong \ZZ[x_1^{\pm 1}, x_2^{\pm 1}]$.  We write $N = \Hom (M, \ZZ)$ for its dual lattice, $M_{\RR} := M\otimes\RR$, $N_{\RR} := N\otimes\RR$, and we denote the standard pairing of $m\in M$ and $n\in N$ by $m \cdot n$.  Given a strictly convex rational cone $\sigma \subsetneq M_{\RR}$, we write $P=P_{\sigma}=\sigma \cap M$.  Let $\widehat{\ZZ[P]}$ denote the completion of the monoid ring $\ZZ[P]$ at the maximal monomial ideal $\fm$ generated by $\left\{x^m \,|\, m\in P\smallsetminus\{0\}\right\}$.

The following are special cases of definitions which originally appeared in
\cite{KS}, \cite{GS}.

\begin{definition}
  \label{walldef}
  A \emph{wall} is a pair $(\fd, f_{\fd})$, where 
  \begin{itemize}

    \item 
      $\fd \subset M_{\mathbb{R}}$ is either a ray $\RR_{\le 0} w$ or a line
      $\RR w$ with $w\in \sigma \cap(M\smallsetminus 0)$;

    \item 
      $f_{\fd} \in \widehat{\ZZ [P]}$ is such that 
      \[ 
        f_{\fd} = f_{\fd}(x^w) = 1 + \sum_{k\geq 1} c_k x^{k w},
      \] 
      for some $c_k \in \ZZ$. 
  \end{itemize}
  The set $\fd \subset M_{\mathbb{R}}$ is called the \emph{support} of the wall
  $(\fd, f_{\fd})$. (Note the different use of the word ``support" in this geometric context.)
\end{definition}

\begin{definition}
  \label{def:scattering_diagram}
  A scattering diagram $\fD$ is a collection of walls such that, for each $k \geq
  0$, the set
  \[
    \{ (\fd, f_{\fd}) \in \fD\, |\, f_{\fd} \neq 1 \bmod \fm^k \}
  \]
  is finite. The support of a scattering diagram is the union of the supports of its walls.
\end{definition}
For simplicity, we will impose the additional condition that no two walls in the scattering diagram have the same support.

Given a wall $(\fd, f_{\fd})$ and a direction $v\in M$ transversal to $\fd$, we associate the element $\fp_{v,\fd}\in
{\mathrm{Aut}}_{\ZZ-alg}\left(\widehat{\ZZ[P]}\right)$ defined by
\[
  \fp_{v,\fd} (x^m) := x^m f_{\fd}^{m\cdot n }, 
\]
where $n\in N$ is the primitive vector annihilating the tangent space to $\fd$ determined by the sign convention $ v\cdot n <0$.  Note
that the only role of the transversal direction $v$ is to fix which of the two
normals $\pm n$ is used in the exponent. 

Let $\fD$ be a scattering diagram.  A path $\gamma: [0,1] \rightarrow M_{\mathbb{R}} \smallsetminus  \{0 \}$ is called \emph{regular with respect to $\fD$} if it is a smooth immersion with endpoints not in the support of $\fD$ which is transverse to each wall of $\fD$ that it crosses. We define the \textit{path-ordered product} $\fp_{\gamma, \fD}$ along such $\gamma$ as follows. For each power $k \geq 1$, let  
\[
  0< t_1 <  t_2 < \cdots < t_s < 1 
\]
be the longest sequence such that $\gamma(t_i)\in\fd_i$ for a wall
$(\fd_i,f_{\fd_i}) \in \fD$ with $f_{\fd_i} \neq 1 \text{ mod } \fm^k$.  In view of the definition of scattering diagrams, such a sequence is finite; we can therefore consider the composition 
\[
  \fp^{(k)}_{\gamma, \fD} :=
  \fp_{\gamma'(t_s),\fd_s} \circ \cdots \circ \fp_{\gamma'(t_1),\fd_1}.
\]
Then we define
\[
  \fp_{\gamma, \fD} := \lim_{k \rightarrow \infty} \fp ^{(k)}_{\gamma, \fD}. 
\]
\begin{definition}
  A scattering diagram is \emph{consistent} if $\fp _{\gamma, \fD}$ depends
  only on the endpoints of $\gamma$ for any path $\gamma$ which is regular 
  with respect to $\fD$.
\end{definition}

\begin{theorem}\cite{KS},\cite{GS}
  \label{th:KS}
  Given any scattering diagram $\fD$, there exists a consistent scattering
  diagram $\fD'$ which contains $\fD$ such that $\fD'\smallsetminus\fD$ only consists
  of rays.
\end{theorem}

We now associate a consistent scattering diagram $\fD_{(b,c)}$ to 
$\cA(b,c)$.  Following \cite[Example~1.30]{GHKK}, we take $\sigma$ to be the second quadrant, i.e., the cone generated by $(-1,0)$ and
$(0,1)$. Define the ``initial'' scattering diagram associated $\cA(b,c)$ as
\[
  \fD_{\mathrm{in},(b,c)} := 
  \left\{
    \big( \RR (-1,0), 1+x_1^{-b}\big), 
    \big( \RR (0,1), 1+x_2^c\big) 
  \right\}.
\]
We then let $\fD_{(b,c)}$ denote the consistent scattering diagram obtained by
applying Theorem \ref{th:KS} to $\fD_{\mathrm{in},(b,c)}$.  The case of $\fD_{(2,1)}$ is illustrated in \cref{fig:diagex}.

\begin{figure}
  \centering
  \begin{tikzpicture}
    \draw (3,0) -- (-3,0) node[left] {$1+x_1^{-2}$};
    \draw (0,-3) -- (0,3) node[above] {$1+x_2$};
    \draw (0,0) -- (3,-3) node[below right] {$1+x_1^{-2}x_2^2$};
    \draw (0,0) -- (3,-1.5) node[below right] {$1+x_1^{-2}x_2$};
  \end{tikzpicture}
  \caption{The scattering diagram $\fD_{(2,1)}$.  Since the initial diagram $\fD_{\mathrm{in},(2,1)}$ consists of the walls $\big(\RR(-1,0),\allowbreak 1+x_1^{-2}\big)$ and $\big(\RR (0,1), 1+x_2\big)$,
  the associated consistent scattering diagram $\fD_{(2,1)}$ contains $\fD_{\mathrm{in},(2,1)}$ together with
  the two walls $\big(\RR_{\leq 0} (-1,1), 1+x_1^{-2}x_2^2 \big)$ and $\big( \RR_{\leq 0} (-2,1), 1+x_1^{-2}x_2 \big)$.
} 
  \label{fig:diagex}
\end{figure}

While this example portrays a scattering diagram with finitely many rays, the diagram $\fD_{(b,c)}$ will consist of an infinite number of rays precisely when $bc\ge 4$. A detailed description of the rays
which appear for $bc\ge4$ can be found in \cite[Example 1.30]{GHKK}.
We summarize the crucial points here.

First of all note that, in view of the definition of scattering diagrams, all the
rays in  $\fD_{(b,c)} \smallsetminus \fD_{\mathrm{in},(b,c)}$ are contained in the
fourth quadrant.  To make our next observation we need to extend the action of
linear operators on $M_\RR$ to an action on pairs $(\fd,f_\fd)$. If $S$ is
linear on $M_\RR$, set
\begin{equation}
  \label{eqn:linear action}
  S(\fd,f_\fd(x^w))
  :=
  \left( S(\fd), f_\fd\left(x^{S(w)}\right) \right).
\end{equation}
Note that, even if $(\fd,f_\fd)$ is a wall, $S(\fd,f_\fd)$ needs not be a wall
since $S(w)$ may lie outside of the cone $\sigma$ (in which case we also get
that $f_\fd\left(x^{S(w)}\right)$ is not an element of $\widehat{\ZZ [P_\sigma]}$, it will actually be contained in $\widehat{\ZZ [P_{S(\sigma)}]}$).

Now consider the two linear involutions $S_1$ and $S_2$ given by
\[
  S_1 =  
  \begin{pmatrix}
    -1 & -b \\
    0& 1
  \end{pmatrix}
  \quad
  \mbox{and}
  \quad
  S_2 =  
  \begin{pmatrix}
    1 & 0 \\
    -c & -1
  \end{pmatrix}.
\]
If $(\fd, f_{\fd}) \in \fD_{(b,c)} \smallsetminus \fD_{\mathrm{in},(b,c)}$ and $S_i
(\fd)$ is contained strictly in the fourth quadrant, then $S_i(\fd, f_{\fd}) \in
\fD_{(b,c)} \smallsetminus \fD_{\mathrm{in},(b,c)}$. Moreover, both 
\begin{equation}
  \label{eq:first-walls}
  S_2\big(\RR_{\leq 0}(-1,0),1+x_1^{-b}\big)
  \quad
  \mbox{and}
  \quad
  S_1\big(\RR_{\leq 0}(0,1),1+x_2^c\big)
\end{equation}
are walls in $\fD_{(b,c)} \smallsetminus \fD_{\mathrm{in},(b,c)}$ even though
neither 
\[
  \big(\RR_{\leq 0}(-1,0),1+x_1^{-b}\big)
  \quad
  \mbox{nor}
  \quad
  \big(\RR_{\leq 0}(0,1),1+x_2^c\big)
\]
is a wall in $\fD_{(b,c)}$. Using \cite[Section 4]{GP} with a change of basis, these considerations gives us a recipe to produce
elements of $\fD_{(b,c)} \smallsetminus \fD_{\mathrm{in},(b,c)}$: it is enough to apply
alternatively $S_1$ and $S_2$ to the walls (\ref{eq:first-walls}).

We need to distinguish three cases. If $bc<4$, this procedure will construct, in
finitely many steps, all the walls in $\fD_{(b,c)} \smallsetminus
\fD_{\mathrm{in},(b,c)}$.  If $bc\ge4$, we will get two infinite
families of walls whose supports will converge respectively to the rays
spanned by the vectors
\[
  \left(2b,-bc + \sqrt{bc(bc-4)}\right)
  \quad
  \mbox{and}
  \quad
  \left(2b,-bc - \sqrt{bc(bc-4)}\right).
\]
These will exhaust all the walls in $\fD_{(b,c)} \smallsetminus \fD_{\mathrm{in},(b,c)}$
with support lying outside the convex cone spanned by these vectors.  When $bc=4$, this cone will be a single rational ray in $\fD_{(b,c)}$.  For $bc>4$, the structure of the remaining part of 
$\fD_{(b,c)}$ is not completely understood; the expectation is that there is
a wall for each possible rational slope inside this irrational cone, partial evidence for this is displayed in Figure~\ref{diagram32} for the case $(b,c)=(3,2)$. 

On the other hand, the chamber structure (i.e. the collection of cones in which the rays cut the plane) one sees outside of the irrational cone
is very well-behaved and familiar in the theory of cluster algebras.  This chamber structure coincides with the Fock-Goncharov cluster complex, see e.g.
\cite[Section 2]{GHKK}, the mutation fan of Reading \cite{R}, and the picture group of Igusa-Orr-Todorov-Weyman \cite{IOTW}.

\begin{figure}
  \input{scatter32_depth100}
  \caption{The supports of all the walls $(\fd,f_\fd)$ in $\fD_{(3,2)}$ such
    that $f_\fd \neq 1 \bmod \fm^{100}$; the boundary rays of
    the irrational cone are highlighted.} 
  \label{diagram32}
\end{figure}

The next result explains how to obtain Laurent polynomials out of scattering diagrams
and serves as the motivation for our later connections to cluster algebras.
\begin{theorem} 
  \label{univLaurent} 
  Let $\fD:=\fD_{(b,c)}$ be as constructed above and consider a Laurent polynomial $f \in \ZZ[M]$. For any path $\gamma$ which is regular with respect to $\fD$, $\fp_{\gamma,\fD}(f)$ can be viewed as an element of
  $\ZZ[[x_1^{-1},x_2]]$ localized at $x_1^{-1}x_2$.  If for any such $\gamma$ in
  $M_{\RR}$, with starting point in the first quadrant and endpoint in one of the
  chambers of $\fD$, we have that $\fp_{\gamma,\fD}(f)$ lies in $\ZZ[M]$, then $f$ is a universal
  Laurent polynomial.
\end{theorem}

\begin{proof}
  This is \cite[Theorem 4.4]{GHKK} applied to the case at hand.
  Specifically, let $\cA$ be the cluster variety defined by the given choice of
  seed. By definition, $\cA$ is obtained by gluing together a collection of tori
  via cluster transformations and thus a regular function on $\cA$ is precisely a
  universal Laurent polynomial.  On the other hand, in
  \cite[Section 4]{GHKK} another variety $\cA'$ is defined. This is done by
  associating a torus $\cA_{\tau}':= \Spec \ZZ[M]$ to a chamber $\tau\subseteq
  M_{\RR}$ of $\fD$. For any two chambers $\tau,\tau'$ we can glue $\cA_{\tau}'$
  to $\cA_{\tau'}'$ using the rational map defined on function fields by
  $\fp_{\gamma,\fD}:\ZZ(x_1,x_2)\rightarrow \ZZ(x_1,x_2)$, where $\gamma$
  is a path beginning in $\tau'$ and ending in $\tau$. Performing these gluings
  gives $\cA'$.

  Now \cite[Theorem 4.4]{GHKK} gives an explicit isomorphism between $\cA$ and
  $\cA'$, and thus the algebra of regular functions on $\cA$ and $\cA'$ are
  isomorphic. Furthermore, this isomorphism restricts to the identity on the
  torus of $\cA$ corresponding to the initial seed and the torus of $\cA'$
  corresponding to the positive chamber. In particular, a function $f$ on this
  torus extends to a function on $\cA'$ if $\fp_{\gamma,\fD}(f)$ lies in
  $\ZZ[M]$ for any path $\gamma$ from the positive chamber to any other
  chamber. This shows the characterization of universal Laurent polynomials.
\end{proof}

We now recall the notion of broken lines, which are tropical analogues of holomorphic disks.  They were introduced in \cite{G10}, their theory was further developed in \cite{CPS}, and they were
used in \cite{GHK11} and \cite{GHKK} to construct canonical bases in various
circumstances.

\begin{definition} 
  \label{brokendef}
  Let $\fD$ be a scattering diagram, $m \in M \smallsetminus \{0\}$, and $q \in
  M_{\RR} \smallsetminus \text{Supp}(\fD)$.  A \emph{broken line} with initial
  \emph{exponent} $m$ and endpoint $q$ is a continuous, piecewise linear path
  $\gamma : ( - \infty , 0] \rightarrow M_{\mathbb{R}} \smallsetminus \{ 0\} $ with
  a finite number of domains of linearity and a choice of monomial
  $c(\ell) x^{m(\ell)} \in \ZZ[M]$ for each domain of linearity $\ell \subseteq ( -
  \infty, 0]$ of $\gamma$. 
  
  The path $\gamma$ and the monomials $c(\ell) x^{m(\ell)}$ need to satisfy the
  following conditions:
  \begin{itemize}
    \item $\gamma(0) = q$;
    \item if $\ell$ is the first (i.e. unbounded) domain of linearity of $\gamma$, then 
      \[c(\ell) x^{m(\ell)} = x^{m};\]
    \item for $t$ in a domain of linearity $\ell$, $\gamma'(t) = -m(\ell)$;
    \item $\gamma$ bends only when it crosses a wall. If $\gamma$ bends from the
      domain of linearity $\ell$ to $ \ell'$ when crossing $(\fd, f_{\fd})$, then
      $c(\ell')x^{m(\ell')}$ is a term in 
      \[\fp_{-m(\ell), \fd} \left(c(\ell) x^{m(\ell)}\right).\]
  \end{itemize}
\end{definition}
We refer to $m(\ell)\in \mathbb{Z}^2$ as the \emph{exponent} of that domain
of linearity.

We are finally ready to introduce the main player of our discussion.
For a broken line $\gamma$ we denote by $\Mono(\gamma)$ the monomial attached to the last domain of linearity of $\gamma$.
\begin{definition}
  Let $\fD, m, q$ be as in Definition~\ref{brokendef}.  Define the \emph{theta function} corresponding to $m$ and $q$
  as
  \[ 
    \vartheta_{q, m} = \sum_{\gamma} \Mono (\gamma), 
  \] 
  where the sum is over all broken lines with initial exponent $m$ and endpoint
  $q$.
\end{definition}

\begin{example} 
  \label{brokenex}
  Consider the scattering diagram $\fD_{(2,2)}$ and let $q$ be a small
  irrational perturbation of the point $(1.5,1)$. There are three broken lines
  with initial exponent $m = (1,-1)$ and endpoint $q$ as shown in
  Figure~\ref{figbrokenex}.
  First of all, we can have a broken line $\gamma_1$ which does not bend.
  Therefore
  \[
    \Mono(\gamma_1) = x_1 x_2^{-1}.
  \]
  There is the broken line $\gamma_2$  which bends only at the $x$-axis. Since
  \[ 
    \fp_{(-1,1), \RR(-1,0)} (x_1 x_2^{-1}) = 
    x_1 x_2^{-1}(1+x_1^{-2}) =  x_1 x_2^{-1} + x_1^{-1} x_2^{-1},
  \]
  to bend we need to choose the second term and obtain 
  \[
    \Mono(\gamma_2) =  x_1^{-1} x_2^{-1}.
  \]
  The last broken line $\gamma_3$ bends both at the $x$- and $y$-axes, the
  latter bend coming from
  \[ 
    \fp_{ (1,1), \RR (0,1)} ( x_1^{-1} x_2^{-1}) =  
    x_1^{-1} x_2^{-1} + x_1^{-1} x_2.  
  \]
  This time we have 
  \[
    \Mono (\gamma_3) = x_1^{-1} x_2.
  \]
  Thus the theta function associated to $m = (1,-1)$ with endpoint point $q$ is 
  \[ 
    \vartheta_{q, (1,-1)} =  
    x_1 x_2^{-1} + x_1^{-1} x_2^{-1} +x_1^{-1} x_2 .  
  \]
\end{example}

\begin{figure}
  \centering
  \begin{tikzpicture}
    \draw (3,0) -- (-3,0) node[left] {$1+x_1^{-2}$};
    \draw (0,-3) -- (0,3) node[above] {$1+x_2^2$};
    \draw (0,0) -- (3,-1.5) node[right] {$1+x_1^{-4}x_2^2$};
    \draw (0,0) -- (1.5,-3) node[below] {$1+x_1^{-2}x_2^4$};
    \draw (1.5,1) node[circle, right] {$q$};
    \draw[dotted] (1.4,-2.8) -- (2.8,-1.4);
    \draw[red] (3,-0.5) -- (1.5,1);
    \draw[red] (2.1,0.1) node[above right]{\small $x_1 x_2^{-1}$};
    \draw[red] (3,-0.5) node[left]{$\gamma_1$};
    \draw[blue] (0.5,0) -- node[left] {\small $x_1^{-1} x_2^{-1}$}  (1.5,1);
    \draw[blue] (0.5,0) -- node[below] {\small $x_1 x_2^{-1}$} (3,-2.5);
    \draw[blue] (3,-2.5) node[left]{$\gamma_2$};
    \draw[green!75!black] (0,2.5)  -- node[right] {\small $x_1^{-1}x_2$} (1.5,1);
    \draw[green!75!black] (-2.5,0) -- node[left] {\small $x_1^{-1} x_2^{-1}$}  (0,2.5);
    \draw[green!75!black] (0.5, -3)  --node[left] {\small $x_1 x_2^{-1}$} (-2.5,0);
    \draw[green!75!black] (0.5,-3) node[above]{$\gamma_3$};
    \draw[color=blue,fill=blue] (0.5,0) circle (0.5mm);
    \draw[color=green!75!black,fill=green!75!black] (0,2.5) circle (0.5mm);
    \draw[color=green!75!black,fill=green!75!black] (-2.5,0) circle (0.5mm);
    \draw[color=black,fill=black] (1.5,1) circle (0.5mm);
  \end{tikzpicture}
  \caption{The scattering diagram $\fD_{(2,2)}$ and the broken lines described in Example~\ref{brokenex}.} 
  \label{figbrokenex}
\end{figure}
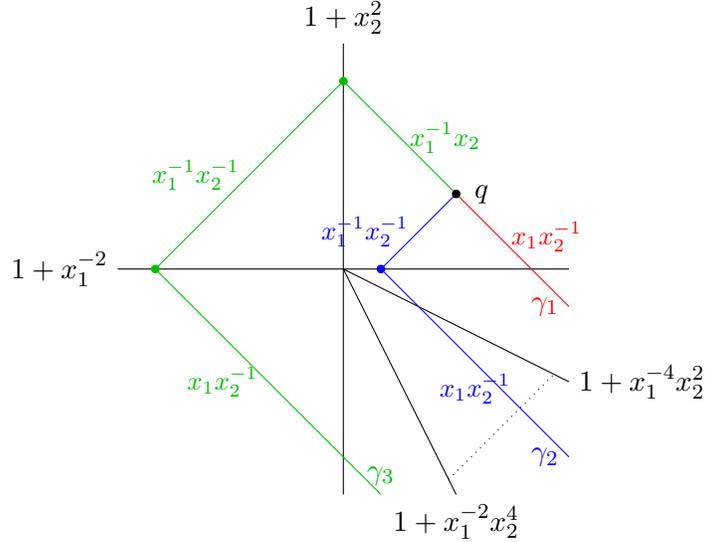

The following summarizes the main properties of the theta functions as shown in
\cite{CPS} and \cite{GHKK}.
\begin{theorem} $ $
  \begin{enumerate}
    \item
      If $\fD$ is any consistent scattering diagram, $q$ and $q'$ are two general
      irrational points on $M_{\RR} \smallsetminus$ Supp$(\fD)$, and $\gamma$ is a
      path joining $q$ to $q'$, then $\fp_{\gamma, \fD }(\vartheta_{q,m}) = \vartheta_{q', m}$. 
    \item Take $\fD=\fD_{(b,c)}$. 
      \begin{enumerate}
        \item If $q$ and $m$ lie in the interior of the same chamber of $\fD$, then $\vartheta_{q,m}=x^{m}$.
        \item If $q$ lies in the interior of a chamber of $\fD$, then $\vartheta_{q,m}$ is a Laurent polynomial for any $m$.
        \item If $q$ lies in the interior of the first quadrant, then $\vartheta_{q,m}$ is a universal Laurent polynomial for any $m$.
      \end{enumerate}
  \end{enumerate}
\end{theorem}

\begin{proof}
  (1) is a main result of \cite{CPS}, see also \cite[Theorem~3.5]{GHKK} for its
application to scattering diagrams in the current context. (2a) is 
\cite[Proposition~3.8]{GHKK} if $q$ and $m$ are both in the positive quadrant
of $M_{\RR}$. If $q$ and $m$ are in some other chamber, say $\sigma$, then
by \cite[Construction~1.38]{GHKK}, there is a scattering diagram $\fD'$
obtained from a mutation of the initial seed defining $\fD$ and a piecewise
linear map $T_v:M_{\RR}\rightarrow M_{\RR}$ which takes the support of
$\fD$ to the support of $\fD'$, and such that the positive chamber of $\fD'$
pulls back to $\sigma$. Furthermore, there is a one-to-one correspondence
between broken lines for $\fD$ and $\fD'$ by \cite[Proposition~3.6]{GHKK}.
Thus the claim follows from \cite[Proposition~3.8]{GHKK} applied to $\fD'$.

(2b) is \cite[Example~7.18]{GHKK}. In slightly more detail, let 
$\Theta\subseteq M$ denote the
set of $m\in M$ for which $\vartheta_{q,m}$ is a Laurent polynomial for
$q$ general in the first quadrant of $M_{\RR}$. By 
\cite[Theorem~7.16,(3)]{GHKK}, $\Theta$ contains all points of $M$
contained in chambers (i.e., the set of points denoted as $\Delta^+_V(\ZZ)$
in \cite[Theorem~7.16,(3)]{GHKK}). Thus in particular, $\Theta$ contains
all integral points in the first three quadrants of $M_{\RR}$. But
by \cite[Theorem~7.16,(4)]{GHKK}, $\Theta$ is closed under addition, and
hence consists of all points in $M$. It then follows that 
$\vartheta_{q,m}$ is a Laurent polynomial for $q$ in any chamber by
\cite[Proposition~7.1]{GHKK}.

Finally, (2c) follows from from (2b) and Theorem \ref{univLaurent}.
\end{proof}

\begin{remark} 
  \label{rk:theta functions g-vector}
  If $m\in M$ lies in one of the chambers of $\fD_{(b,c)}$ and $q$ lies in
  the first quadrant, then from (1) and (2a) above we see that
  $\vartheta_{q,m}=\fp_{\gamma,\fD_{(b,c)}}(x^{m})$ for a path $\gamma$
  joining the chamber containing $m$ to $q$. Moreover, it follows from the
  details of the proof of Theorem \ref{univLaurent} that $\vartheta_{q,m}$ is
  a cluster monomial and then from \cite[Theorem~7.5]{GHKK} that the 
  \emph{$\bfg$-vector} of this cluster monomial is precisely $m$. We again refer to
  \cite{FZ4} for the definitions and basic properties of $\bfg$-vectors.
\end{remark}

\begin{example}
  Let us try one more calculation with broken lines. We take the same scattering
  diagram as in Example~\ref{brokenex}. Now take the initial exponent $m=(2,-2)$
  with the same endpoint $q$. By similar calculations we get 
  \[ 
    \vartheta _{q, (2,-2)} = 
    x_1^2 x_2^{-2} + x_1^{-2}x_2^2 + x_1^{-2}x_2^{-2} + 2 x_2^{-2} + 2x_1^{-2}.  
  \]
  Note that 
  \[ 
    \vartheta _{q, (2,-2)} = 
    \left(\vartheta_{q, (1,-1)}\right) ^2 -2. 
  \]
  In the scattering diagram $\fD_{(2,2)}$ considered here, the ray with exponent $(1,-1)$ does
  not lie in the interior of any chamber. So neither $\vartheta_{q,(1,-1)}$ nor 
  $\vartheta _{q, (2,-2)}$ is a cluster monomial.

  There are a number of known bases for $\cA(2,2)$ (see \cite{Dup,MSW,LLZ}) which all prescribe different elements
  having $\bfg$-vectors $(d,-d)$ for $d>0$. The
  calculations above show that at least for $d=1$ or $2$, theta functions agree with
  the greedy basis elements.
\end{example}

\section{From $\bfg$-vectors to $\bfd$-vectors}
As mentioned in Remark \ref{rk:theta functions g-vector}, theta functions are
parametrized by their $\bfg$-vectors. On the other hand the description of
greedy elements given in \cite{LLZ} is in terms of their $\bfd$-vectors (cf.
Remark 1.9 ibid.). 

In order to compare the two we will leverage the observation that, in rank 2,
these families of vectors are related by an easy piecewise-linear transformation
as explained in the paragraph following Conjecture 3.21 in \cite{RS}. We will do
so via a scattering diagram $\fD^{\bfd}_{(b,c)}$ closely related to
$\fD_{(b,c)}$.

Let $T:M_\RR\rightarrow M_\RR$ be the piecewise-linear map given by
\[
  T (m) := 
  \begin{cases}
    m   & m_2 \geq 0 \\
    m + (bm_2,0), & m_2 \leq 0.
  \end{cases}
\]
We will denote its domains of linearity by 
\[ 
  \mathcal{H}_{+} := 
  \left\{ m \in M_{\mathbb{R}}\, |\, m_2 \geq 0 \right\} 
  \quad
  \mbox{and}
  \quad
  \mathcal{H}_{-} := 
  \left\{ m \in M_{\mathbb{R}} \,|\, m_2 \leq 0 \right\}.
\]
Let $T_+$ and $T_-$ be the linear extensions to $M_\RR$ of $T|_{\mathcal{H}_+}$
and $T|_{\mathcal{H}_-}$ respectively ($T_+$ is just the identity map but it
will be convenient to use this notation in what follows). 
By (\ref{eqn:linear action}), both $T_+$ and $T_-$ act on pairs $(\fd,f_\fd)$ so
we can use them to define the image of such pairs under $T$. Namely set
\[
  T(\fd,f_\fd):=
  \left\{
    T_+\left(\fd\cap\mathcal{H}_+,f_\fd\right),
    T_-\left(\fd\cap\mathcal{H}_-,f_\fd\right)
  \right\}.
\]

Having fixed the notation we are ready to introduce $\fD_{(b,c)}^\bfd$.
The set
\[
  T(\fD_{(b,c)}):=
  \bigcup_{(\fd,f_\fd)\in \fD_{(b,c)}}
  T(\fd,f_\fd)
\]
is not a scattering diagram according to Definition 
\ref{def:scattering_diagram} (not all of its elements are walls for the same
convex cone), but can be made into one by a few simple fixes.

First of all, $\big( \RR (0,1), 1+x_2^c\big)$ is the only wall of $\fD_{(b,c)}$
whose support is not totally contained in one of the domains of linearity of
$T$; therefore, under $T$, it breaks into two parts:
\[
  \big( \RR_{\ge0} (0,1), 1+x_2^c\big)
  \quad
  \mbox{and}  
  \quad
  \big( \RR_{\le 0} (b,1), 1+x_1^{bc}x_2^c\big).
\]

Next note that, since $T(-1,0)=(-1,0)$ and $T(b,-1)=(0,-1)$, $T$ maps all the
walls of $\fD_{(b,c)}\smallsetminus\fD_{\mathrm{in},(b,c)}$ to the third quadrant.
Indeed, $\big(\RR_{\leq0}(-b,1),1+x_1^{-bc}x_2^c\big)$ is the wall with the biggest slope in
$\fD_{(b,c)}\smallsetminus\fD_{\mathrm{in},(b,c)}$ and its image is $\big( \RR_{\le0} (0,1), 1+x_2^c\big)$.

\begin{definition}
  $\fD_{(b,c)}^\bfd$ is the scattering diagram obtained from
  $T\left(\fD_{(b,c)}\right)$ by replacing 
  \begin{itemize}
    \item
      $\big(\RR (-1,0), 1+x_1^{-b}\big)$ with $\big(\RR  (1,0), 1+x_1^b\big)$,
    \item 
      both $ \big( \RR_{\ge0} (0,1), 1+x_2^c\big)$ and $\big( \RR_{\le0} (0,1),
      1+x_2^c\big)$ with $ \big( \RR (0,1), 1+x_2^c\big)$.
  \end{itemize}
  Its base region is the cone $\sigma^\bfd$ generated by $(1,0)$ and $(0,1)$. 
\end{definition}

\begin{remark}
  It is not too hard to see that the scattering diagram $\fD_{(b,c)}^\bfd$ is
  consistent. This fact, together with the uniqueness property implied by
  \cite[Theorem 1.7]{GHKK}, gives an alternative way to introduce it. Indeed, in
  analogy with the definition of $\fD_{(b,c)}$, one could consider the
  scattering diagram
  $\fD_{\mathrm{in},(b,c)}^\bfd$ given by 
  \[
    \fD_{\mathrm{in},(b,c)}^\bfd=
    \left\{
      \big(\RR (1,0), 1+x_1^b\big), 
      \big(\RR (0,1), 1+x_2^c\big)
    \right\}
  \]
  and obtain $\fD_{(b,c)}^\bfd$ using Theorem \ref{th:KS}.  The case of $\fD_{(2,1)}^\bfd$ is illustrated in Figure~\ref{fig:diagex2}.
\end{remark}

\begin{figure}[h]
  \centering
  \begin{tikzpicture}
    \draw (-3,0) -- (3,0) node[right] {$1+x_1^2$};
    \draw (0,-3) -- (0,3) node[above] {$1+x_2$};
    \draw (0,0) -- (-3,-3) node[below left] {$1+x_1^2x_2^2$};
    \draw (0,0) -- (-3,-1.5) node[below left] {$1+x_1^2x_2^1$};
  \end{tikzpicture}
  \caption{The scattering diagram $\fD_{(2,1)}^\bfd$.}
  \label{fig:diagex2}
\end{figure}

For a broken line $\gamma$ in $\fD_{(b,c)}$, we denote its image under $T$ as
$T(\gamma)$: this is the broken line in $\fD_{(b,c)}^\bfd$ whose underlying map
is $T\circ\gamma$. Given any domain of linearity $\ell$ of $\gamma$, by
subdividing it when necessary, we can always assume that either $\gamma(\ell)
\subset \mathcal{H}_{+} $ or $\gamma(\ell)\subset \mathcal{H}_{-}$. The monomial
attached to $\ell$ in $T(\gamma)$ is then obtained by applying, accordingly, either
$T_+$ or $T_-$ to the exponent of the monomial attached to $\ell$ in $\gamma$.

\begin{theorem}
  \label{theorem:T_on_broken_lines}
  The map $T$ defines a one-to-one correspondence from broken lines in $\fD_{(b,c)}$
  with exponent $m$ and endpoint $q$ to broken lines in $\fD_{(b,c)}^\bfd$ with exponent $T(m)$
  and endpoint $T(q)$. In particular, for $q\in\mathcal{H}_+$ or $q\in\mathcal{H}_-$,
  we have
  \[ 
    \vartheta^\bfd_{T(q),T(m)}=T_{+} \left(\vartheta_{q,m}\right) 
    \quad
    \mbox{or} 
    \quad
    \vartheta^\bfd_{T(q),T(m)}=T_{-}\left(\vartheta_{q,m}\right)
  \]
  respectively.
\end{theorem}

\begin{proof}
  This is essentially the same as the argument of \cite[Proposition~3.6]{GHKK}.  To
  prove the statement, we only need to check the bending at the $x$-axis. Let
  $\ell$, $\ell'$ be the domains of linearity of $\gamma$ before and after bending
  along $\RR (-1,0)$. So $c(\ell') x^{m(\ell')}$ is a term in 
  \[
    \fp_{-m(\ell),\RR(-1,0)} \left(c(\ell) x^{m(\ell)}\right)
    = 
    c(\ell) x^{m(\ell)} \left(1+x_1^{-b}\right) ^{|m_2(\ell)|}.
  \]

  First, assume $\gamma$ passes from $\mathcal{H}_-$ to $\mathcal{H}_+$. In this
  case, we have $m_2(\ell) < 0$. Now in order for the monomial
  $c(\ell')x^{T_+(m(\ell'))} =c(\ell')x^{m(\ell')}$ attached to $\ell'$ in
  $T(\gamma)$ to satisfy the bending rule, it must be a term in
  \[
    \fp_{-T_-(m(\ell)),\RR (1,0)} \left(c(\ell) x^{T_-(m(\ell))}\right). 
  \]
  Since the second component of $T_-(m(\ell))$ is $m_2(\ell)$, we get
  \begin{align*} 
    \fp_{-T_-(m(\ell)),\RR (1,0)} \left(c(\ell) x^{T_-(m(\ell))}\right) 
    &=c(\ell) x^{T_-(m(\ell))} \left(1+x_1^b\right) ^{-m_2(\ell)}\\
    &=c(\ell) x^{m(\ell)} x_1^{b m_2(\ell)}\left(1+x_1^b\right)^{-m_2(\ell)}\\
    &=c(\ell) x^{m(\ell)} \left(1+x_1^{-b}\right) ^{-m_2(\ell)}.
  \end{align*}
  This shows that $T(\gamma)$ satisfies the correct rule when bending along
  $(\RR (1,0), 1+x_1^b)$ if $\gamma$ passes from $\mathcal{H}_-$ to
  $\mathcal{H}_+$. By repeating similar calculations, we can see that this also
  holds when $\gamma$ passes from $\mathcal{H}_+$ to $\mathcal{H}_-$.
\end{proof}

The following demonstrates the utility of using $\fD_{(b,c)}^\bfd$.
\begin{proposition}
  For any  $m\in M$, if $q$ lies in the first quadrant, then 
  \[
    \vartheta^\bfd_{q, m}=x^{m}\left(1+f(x_1,x_2)\right)
  \]
  where $f\in (x_1,x_2)\subseteq \Bbbk[x_1,x_2]$.
  In particular, $m$ is the negative of the $\bfd$-vector of
  $\vartheta^\bfd_{q,m}$.
\end{proposition}

\begin{proof}
  For any $m\in M$ and any $q$ in the first quadrant, there is always a broken
  line $\gamma$ for $m$ and $q$ that does not bend at any wall. Therefore $\Mono
  (\gamma) = x^{m}$ always appears as a term in $\vartheta^\bfd_{q,m}$.

  However, because the functions attached to the walls of
  $\fD_{(b,c)}^\bfd$ are all of the form $1+g(x_1,x_2)$ with $g(x_1,x_2) \in
  (x_1,x_2) \subseteq \Bbbk[[x_1,x_2]]$, it follows that any term coming from a
  broken line which bends must be of the form $cx^{m}x_1^{d_1}x_2^{d_2}$ with
  $d_1,d_2\ge 0$, $d_1+d_2>0$. This proves the result.
\end{proof}

\begin{remark}
  Combining Theorem \ref{theorem:T_on_broken_lines} with the above result, when $q$
  is in the first quadrant we obtain the parametrization of
  theta functions we were after. Indeed, we get 
  \[
    \vartheta^\bfd_{q,T(m)}=\vartheta_{q,m}
  \]
  with $m$ being its $\bfg$-vector and $T(m)$ the negative of its 
  $\bfd$-vector.
\end{remark}

\section{Proof that the bases coincide}

We may now state the main theorem in our current notation.

\begin{theorem}\label{theorem: main}
For any integers $b,c>0$, for each $m=(m_1,m_2)\in \mathbb{Z}^2$, and for each
generic point $q$ in the first quadrant, we have that 
\[ \vartheta_{q,m}^{\mathbf{d}} = x[-m_1,-m_2]\]
as elements in the cluster algebra $\mathcal{A}(b,c)$.  Hence, the greedy basis and the theta basis for $\mathcal{A}(b,c)$ coincide.
\end{theorem}

\noindent The proof will be to show that the support of $\vartheta_{q,m}^{\mathbf{d}}$ is contained in the polygon $R_{m_1,m_2}$ in Theorem \ref{th:Newton polygons}.  By Scholium \ref{sch}, this is already enough to show that $\vartheta_{q,m}^\mathbf{d}=x[-m_1,-m_2]$.

We begin our analysis by describing the ``changes of direction'' of a broken
line $\gamma$ in $\mathfrak{D}_{(b,c)}^\mathbf{d}$. 
Let $\ell$ be a domain of
linearity of $\gamma$. We say that $\gamma$ \emph{moves right} (resp. \emph{up}) in
$\ell$ if $m_1(\ell)<0$ (resp. $m_2(\ell)<0$). Conversely we will say that
$\gamma$ \emph{moves left} or \emph{down} in $\ell$. 

\begin{lemma}
  \label{lemma: change of direction}
  Let $\ell$ and $\ell'$ be two consecutive domains of linearity of a broken line
  $\gamma$ in $\mathfrak{D}_{(b,c)}^\mathbf{d}$. Then
  \[
    m_1(\ell) \leq m_1(\ell')
    \qquad
    \mbox{and}
    \qquad
    m_2(\ell) \leq m_2(\ell')
    .
  \]
\end{lemma}
\begin{proof}
  Suppose $\gamma$ bends along the wall
  $(\fd, f_{\fd})$ when passing from $\ell$ to $\ell'$ then 
  $c(\ell')x^{m(\ell')}$ is a term in
  \[
    \fp_{-m(\ell), \fd} \left(c(\ell) x^{m(\ell)}\right) 
    =
    c(\ell)x^{m(\ell)}f_{\fd}^{m(\ell)\cdot n}
  \]
  with $m(\ell)\cdot n >0$. The desired property then follows immediately from
  the observation that, by how $\mathfrak{D}_{(b,c)}^\mathbf{d}$ has been
  constructed, all the exponents of the monomials of $f_{\fd}$ are 
  non-negative.
\end{proof}

An immediate consequence of this lemma is that, once a broken line begins to move
left or down, it will continue to do so. In particular, if $\gamma$ is a broken
line ending in the first quadrant, it can move left (resp. down) only
in the first and fourth (resp. second) quadrant.

At any point $q=(q_1,q_2)\in \gamma$ at which $\gamma$ is linear with exponent $m=(m_1,m_2)$, define the \emph{angular momentum} of $\gamma$ at $q$ to be $q_2m_1-q_1m_2$.

\begin{lemma}
The angular momentum is constant on $\gamma$.
\end{lemma}
\begin{proof}
Let $q$ and $q'$ be two points on $\gamma$.
First, assume that $q=(q_1,q_2)$ and $q'=(q_1',q_2')$ are in the same linear region of $\gamma$, with exponent $m=(m_1,m_2)$.  Since $\gamma'=-m$ at $q$, there is some $t$ such that 
\[ (q_1',q_2') = (q_1+tm_1,q_2+tm_2) \]
Then the angular momentum at $q'$ is 
\[ (q_2+tm_2)m_1-(q_1+tm_1)m_2 = q_2m_1 - q_1m_2 \]

Next, assume that $q$ and $q'$ are points on $\gamma$ on either side of a bend at a wall $(\fd, f_{\fd})$ at point $q''=(q_1'',q_2'')$.  If the exponent of $\gamma$ at $q$ is $m=(m_1,m_2)$ and $f_{\fd}$ is a series in $x^{(w_1,w_2)}$, then the exponent of $\gamma$ at $q'$ must be of the form $(m_1+kw_1,m_2+kw_2)$ for some positive integer $k$.  By the argument of the previous paragraph, the angular momentum at $q$ is $q_2''m_1-q_1''m_2$
and the angular momentum at $q'$ is
\[ q_2''(m_1+kw_1)-q_1''(m_2+kw_2) = (q_2''m_1-q_1''m_2)+k(q_2''w_1-q_1''w_2) \]
Since the point $(q_1'',q_2'')$ lies on the ray through $(w_1,w_2)$, the expression $q_2''w_1-q_1''w_2$ is zero, and so the angular momenta at $q$ and $q'$ are the same.  This equality extends transitively to any pair of points $q,q'$ on $\gamma$.
\end{proof}
The sign of the angular momentum is a useful invariant for characterizing the qualitative behavior of a broken line.
For a broken line ending in the first quadrant, the sign of the angular momentum characterizes whether that broken line could have passed through the fourth quadrant (positive) or the second quadrant (negative).

\begin{lemma}
Let $\gamma$ be a broken line $\mathfrak{D}_{(b,c)}^{\mathbf{d}}$ with endpoint $q$ in the first quadrant. If $\gamma$ has positive (resp. negative) angular momentum, then the slope of the linear domains of $\gamma$ decreases (resp. increases) at each bend, except possibly at the boundary of the first quadrant. 
\end{lemma}

\tikzstyle{wall}=[draw=black!25,thick]
\tikzstyle{dot} = [blue,fill=blue,inner sep=0.25mm,circle,draw,minimum size=.5mm]

\noindent Figure \ref{fig: brokenlineexample} depicts a broken line with positive angular momentum.  The slopes of the linear domains decrease from $\frac{5}{4}$ to $1$ to $\frac{1}{2}$ before increasing to $+\infty$.

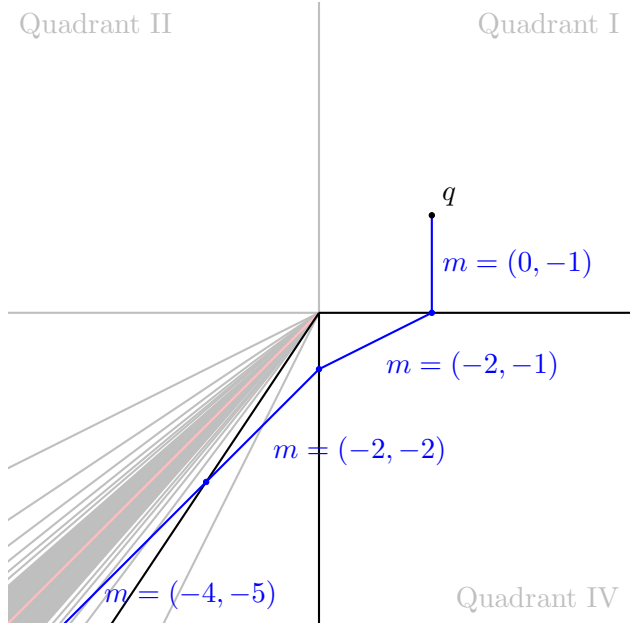
\begin{figure}[h!t]
\begin{tikzpicture}
    \begin{scope}[scale=.75]
      \clip (-5.5,-5.5) rectangle (5.5,5.5);
        \draw[wall] (0,0) to (0,10);
        \draw[wall] (0,0) to (-10,0);
        
        \draw[wall] (0,0) to (-12,-6);
        \draw[wall] (0,0) to (-12,-8);
        \draw[wall] (0,0) to (-12,-9);
        \draw[wall] (0,0) to (-12,-9.6);
        \draw[wall] (0,0) to (-12,-10);
        \draw[wall] (0,0) to (-12,-10.28);
        \draw[wall] (0,0) to (-6,-12);
        \draw[wall] (0,0) to (-9,-12);
        \draw[wall] (0,0) to (-9.6,-12);
        \draw[wall] (0,0) to (-10,-12);
        \draw[wall] (0,0) to (-10.28,-12);        
        \path[fill=black!25] (0,0) to (-10.5,-12) to (-12,-10.5) to (0,0);
        \draw[wall, draw=red!25] (0,0) to (-10,-10);
        \draw[wall,black] (0,0) to (10,0);
        \draw[wall,black] (0,0) to (0,-10);
        \draw[wall,black] (0,0) to (-8,-12);
        
        \draw[blue,thick] (2,1.73) to node[right] {$m=(0,-1)$} (2,0) to node[below right] {$m=(-2,-1)$} (0,-1) to node[below right] {$m=(-2,-2)$} (-2,-3) to node[right] {$m=(-4,-5)$} (-6,-7);
        \node[dot] at (2,0) {};
        \node[dot] at (0,-1) {};
        \node[dot] at (-2,-3) {};
        \node[dot,draw=black, fill=black] (q) at (2,1.73) {};
        \node[above right] at (q) {$q$};
        
        \node[below left,black!25] at (5.5,5.5) {Quadrant I};
        \node[below right,black!25] at (-5.5,5.5) {Quadrant II};
        \node[above left,black!25] at (5.5,-5.5) {Quadrant IV};
        
    \end{scope}
\end{tikzpicture}
\caption{A broken line with positive angular momentum}
\label{fig: brokenlineexample}
\end{figure}

\begin{proof}
The lemma is straightforward except for broken lines with initial exponent $(m_1,m_2)$ with $m_1,m_2<0$.
Consider a bend of $\gamma$ at a point $(q_1,q_2)$ in a wall $(\fd,f_{\fd}(x^{(w_1,w_2)}))$.  If the exponent immediately before the bend is $(m_1,m_2)$, the exponent immediately after the bend is $(m_1+kw_1,m_2+kw_2)$ for some positive integer $k$.  

Assume that $(q_1,q_2)$ is not in the boundary of the first quadrant, so that
$(q_1,q_2)$ is a negative scalar multiple of the exponent $(w_1,w_2)$. By this
assumption, in view of Lemma \ref{lemma: change of direction} and the fact that
$q$ lies in the first quadrant, we have also
$m_1+kw_1,m_2+kw_2 <0$.

If the angular momentum $q_2m_1-q_1m_2$ is positive, then the cross-product $w_2m_1-w_1m_2$ is negative.  But for positive $k$,
\[ k(w_2m_1-w_1m_2) =m_1(m_2+kw_2)-m_2(m_1+kw_1)<0 \Rightarrow \frac{m_2+kw_2}{m_1+kw_1} < \frac{m_2}{m_1}\]
as desired.  If the angular momentum is negative, the slope increases by an identical argument.
\end{proof}

We can now constrain the possible final exponent of a broken line, which will be used to bound the support of the corresponding theta function.

\begin{lemma}
Let $\gamma$ be a broken line in $\mathfrak{D}_{(b,c)}^{\mathbf{d}}$ which begins in the third quadrant, with endpoint $q$ in the first quadrant.  Denote the  initial exponent by $m=(m_1,m_2)$ and the final exponent by $m^q=(m_1^q,m_2^q)$.  
\begin{enumerate}
  \item If $\gamma$ has positive angular momentum, then $m_2\leq m_2^q<0$ and
  \[ m_1\leq m_1^q\leq\left( \frac{m_1}{m_2}-b\right)m_2^q\]
  where the upper bound is equality only when $m^q =( m_1-bm_2,m_2)$.
  \item If $\gamma$ has negative angular momentum, then $m_1\leq  m_1^q<0$ and
  \[ m_2 \leq m_2^q\leq\left( \frac{m_2}{m_1}-c\right)m_1^q\]
  where the upper bound is equality only when $m^q =( m_1,m_2-cm_1)$.
\end{enumerate}
\end{lemma}
\begin{proof}
Assume $\gamma$ has positive angular momentum; consequently, $\gamma$ passes through the fourth quadrant before entering the first quadrant.  Let $(m_1',m_2')$ be the exponent on $\gamma$ in the fourth quadrant.  By the preceding lemma, $\frac{m_2'}{m_1'}\leq \frac{m_2}{m_1}$ with equality only if $\gamma$ doesn't bend before it reaches the fourth quadrant.  

As the broken line passes into the first quadrant, it may bend at the wall $(\mathbb{R}(1,0),1+x^{(b,0)})$.  By definition, the final exponent $m^q$ on $\gamma$ must be an exponent that appears in $x^{(m_1',m_2')}(1+x^{(b,0)})^{-m_2'}$.  It follows that 
\[ (m_1^q,m_2^q) = (m_1'+kb,m_2')\]
for some $0\leq k\leq -m_2'$.
Consequently, $m_2^q=m_2'<0$ and 
\[ m_1^q\leq m_1' - bm_2' = \left(\frac{m_1'}{m_2'} - b\right) m_2'  = \left(\frac{m_1'}{m_2'} - b\right) m_2^q \leq \left(\frac{m_1}{m_2} - b\right) m_2^q\]
The second inequality is equality only if $(m_1',m_2')=(m_1,m_2)$, and so the composite inequality is equality only if $m^q =( m_1-bm_2,m_2)$.

Analogous inequalities hold for negative angular momentum by the same argument.
\end{proof}

\begin{proof}[Proof of Theorem \ref{theorem: main}]

If $m=(m_1,m_2)$ such that $m_1\geq0$ or $m_2\geq0$, then $\vartheta_{q,m}^\mathbf{d}$ is the cluster monomial $x[-m_1,-m_2]$, by Remark \ref{rk:theta functions g-vector}.
Next, assume that $m=(m_1,m_2)$ such that $m_1\leq0$ and $m_2\leq0$.
The coefficient of $x^{(a_1,a_2)}$ in $\vartheta_{q,m}^\mathbf{d}$ can have non-zero coefficient only if there is a broken line $\gamma$ in $\mathfrak{D}^\mathbf{d}_{(b,c)}$ with initial exponent $m$ and final exponent $a=(a_1,a_2)$.  By the preceding lemma, this implies that 
\[ m_1\leq a_1\leq \left(\frac{m_1}{m_2}-b\right)a_2,\;\;\; m_2\leq a_2\leq \left(\frac{m_2}{m_1}-c\right)a_1\]
Furthermore, the upper bounds are only satisfied in the specific cases when $(a_1,a_2)$ is equal to $(m_1-bm_2,m_2)$ or $(m_1,m_2-cm_1)$.  
Since $\vartheta_{q,m}^\mathbf{d}\in \mathcal{A}_{b,c}$, Scholium \ref{sch} implies that $\vartheta_{q,m}^\mathbf{d}$ is a scalar multiple of $x[-m_1,-m_2]$. 

To show they coincide, we consider the coefficient of $x^{(m_1,m_2)}$ in each element.  The coefficient of $x^{(m_1,m_2)}$ in $x[-m_1,-m_2]$ is $1$, by the definition of a pointed element. The coefficient of $x^{(m_1,m_2)}$ in $\vartheta_{q,m}^\mathbf{d}$ is the sum of the coefficients of all broken lines in $\mathfrak{D}_{(b,c)}^\mathbf{d}$ with initial exponent $(m_1,m_2)$ and final exponent $(m_1,m_2)$.  Since any bend in a broken line would increase one of the components of the exponent, this only happens for the unique broken line with initial exponent $(m_1,m_2)$ that has no bends. Hence, the coefficient of $x^{(m_1,m_2)}$ in $\vartheta_{q,m}^\mathbf{d}$ is $1$, and so $\vartheta_{q,m}^\mathbf{d}=x[-m_1,-m_2]$.
\end{proof}

\begin{remark}
As mentioned in Remark \ref{remark:Dyck}, the coefficients of $x[-m_1,-m_2]$ may be interpreted as counting `compatible pairs' in a lattice path called a `maximal Dyck path'.  One consequence of Theorem \ref{theorem: main} is that the coefficients $c(p,q)$ are equal to a weighted sum of certain broken lines.  An interesting open problem is to reprove the coincidence of the two bases by giving a combinatorial bijection between broken lines and compatible pairs which directly proves the equality of the respective coefficients.
\end{remark}

\end{document}